\documentclass[12pt]{article}
\usepackage{latexsym}
\usepackage{graphicx}
\usepackage{verbatim}
\usepackage{amssymb}
\usepackage{amsmath}
\usepackage{amsthm}
\title{Strong Marstrand theorems and dimensions of sets formed by subsets of hyperplanes} 
\author{Kenneth Falconer and Pertti Mattila}

\date{}

\newtheorem{theo}{Theorem}

\newtheorem{lem}[theo]{Lemma}

\newcommand{\be}{\begin{equation}} 
\newcommand{\ee}{\end{equation}} 
\allowdisplaybreaks

\renewcommand{\dh}{\dim_{\rm {H}}}

\newcommand{\proj}{ \mbox{\rm proj}}

\newcommand{\cH}{\mathcal{H}}
\newcommand{\cL}{\mathcal{L}}
\newcommand{\R}{\mathbb{R}}
\newcommand{\N}{\mathbb{N}}

\begin{document}
\maketitle
\abstract{\noindent  We present strong versions of Marstrand's projection theorems and other related theorems. For example, if $E$ is a plane set of positive and finite $s$-dimensional Hausdorff measure, there is a set $X$ of directions of Lebesgue measure $0$, such that the projection onto any line with direction outside $X$, of any subset $F$ of $E$ of positive $s$-dimensional measure, has Hausdorff dimension $\min\{1,s\}$, i.e. the set of exceptional directions is independent of $F$. Using duality this leads to results on the dimension of sets that intersect families of lines or hyperplanes in positive Lebesgue measure.

\section{Introduction}
\setcounter{equation}{0}
\setcounter{theo}{0}

A Besicovitch set is a subset of $\R^n$ of Lebesgue measure zero which contains a unit line segment in every direction. Questions related to their Hausdorff dimension are connected with many problems of modern Fourier analysis and have been studied extensively during the last quarter of century, see, for example, \cite{Mat2}. An elegant way to construct such sets, going back to Besicovitch, is based on duality between lines and points. This in fact gives more, namely sets of measure zero which contain an entire line in every direction. Keleti \cite{Kel}  considered the question of whether there is any difference in the Hausdorff dimension for these two types of sets. He showed that in the plane there is not: any union of line segments in the plane has the same Hausdorff dimension as the corresponding union of lines. In this paper we shall show, with mild measurability assumptions, that more is true: we can replace line segments with subsets of lines with positive length. We shall formulate and prove this for hyperplanes in $\R^n$.

Our principal tool will be a strong version of Marstrand's projection theorem. The basic Marstrand theorem, see \cite{Mar, Mat1}, tells us that if $E$ is an $s$-subset of the plane, that is, if $E$ is measurable with respect to the $s$-dimensional Hausdorff measure $\mathcal H^s$ with 
$0<\mathcal H^s(E)<\infty$, then the projection of $E,~ \proj_L E,$ on almost every line $L$ through the origin has Hausdorff dimension $\min\{s,1\}$. Moreover, if $s>1$, then almost all projections of $E$ have positive length. In fact, Marstrand \cite{Mar} proved a more general result when $s>1$ which seems to have been almost forgotten. He showed that for almost all projections the length of $\proj_L F$ is positive for  all $\mathcal H^s$ measurable subsets $F$ of $E$ with $\mathcal H^s(F)>0$, that is, the exceptional set of lines is independent of $F$. We shall give a simple proof for this and we shall also prove the corresponding strong Marstrand theorem in the case $s\leq1$. We shall first formulate and prove these results for projections on $m$-planes in $\R^n$ and then extend them to strong versions of related theorems, including bounds on the dimension of the exceptional set of directions and on `generalized projections' subject to a transversality condition.

We would like thank the referee for the careful reading of the paper and for many valuable comments.

\section{Strong Marstrand theorems}
\setcounter{equation}{0}
\setcounter{theo}{0}

We denote by $G_{n,m}$ the Grassmannian manifold of $m$-dimensional linear subspaces of $R^n$ and by $\gamma_{n,m}$ its orthogonally invariant Borel probability measure. We write $\cL^m$ to denote $m$-dimensional Lebesgue measure on any $m$-dimensional plane. We denote by $\proj_V$  orthogonal projection onto a linear subspace $V$ of $\R^n$ and by $\proj_V\mu$ the image of a measure $\mu$ under $\proj_V$ defined by $(\proj_V\mu)(A) = \mu(\proj_V^{-1}A)$.

\begin{theo}\label{marthm}
Let $1\leq m \leq n-1$ and let $E \subset \mathbb{R}^n$ be an $\mathcal H^s$-measurable set with $0<\cH^s(E) < \infty$, where $s=\dh E>0$. Then there exists a set $X\subset G_{n,m}$ with $\gamma_{n,m}(X) = 0$ such that for all $V \in G_{n,m}\setminus X$ and all 
$\mathcal H^s$-measurable sets $F\subset E$ with $\cH ^s(F)>0$,

$({\rm i})$ $\dh \mbox{\rm proj}_V F = \min\{\dh F, m\}$,
\smallskip

$({\rm ii})$  if $s >m$ then ${\mathcal L}^m (\mbox{\rm proj}_V F) > 0$.
\smallskip

\end{theo}

\begin{proof}
(i) Since $\proj_V$ is a Lipschitz mapping, $\dh \mbox{\rm proj}_V F \leq \min\{\dh F, m\}$ for all $V\in G_{n,m}$ follows immediately.

A set $E$ with $0<\cH^s(E) < \infty$ satisfies an upper density bound $\lim_{r\to 0} \cH^s(E \cap {B(x,r)})/(2r)^s \leq 1$ for $\cH^s$-almost all $x\in E$, see \cite[Corollary 2.5]{Fal3} or \cite[Theorem 6.2]{Mat1}. Thus we may find a countable disjoint collection $E_i$ of $\mathcal H^s$-measurable subsets of $E$ with $\cH^s(E\setminus\bigcup_{i=1}^\infty E_i )=0 $ and numbers $c_i>0$  such that 
\begin{equation}\label{density}
\cH^s(E_i \cap {B(x,r)}) \leq c_i r^s \mbox{ for all } x \in E_i, r>0.
\end{equation}
For each $i$, define the restriction $\mu_i$ of $\mathcal H^s$ to $E_i$ by $ \mu_i (A) = \cH^s(E_i\cap A)$, so that 
\begin{equation}\label{sum}
\cH^s|_{E}= \sum_{i=1}^\infty\mu_i.
\end{equation}

Let $0<t<s$.  Then each $\mu_i$   has finite $t$-energy: 
$$I_t(\mu_i) := \int\int \frac{d\mu_i(x)d\mu_i(y)}{ |x-y|^{t} }< \infty;$$
this follows by integrating the energy integral by parts with respect to one of the $\mu_i$ and using \eqref{density}, see \cite[page 78]{Fal3} or \cite[page 109]{Mat1}.
Just as in the usual potential theoretic proof of the projection theorems,
\begin{align}
\int_{V \in G_{n,m}}I_t(\proj_V\mu_i)d \gamma_{n,m}(V)&= \int_{V \in G_{n,m}}\bigg[ \int_V \int_V
\frac{d(\proj_V\mu_i)(u)d(\proj_V\mu_i)(v)}{|u-v|^t} \bigg]d \gamma_{n,m}(V)\nonumber\\
&\leq c_{n,m,t} \int\int\frac{d \mu_{i}(x)d\mu_i(y)}{|x-y|^t}  < \infty.\label{enineq}
\end{align}
In particular, $\proj_V\mu_i$   has finite $t$-energy for all $V  \in G_{n,m}\setminus X_{i,t}$, where  $X_{i,t}$ is a subset of $G_{n,m}$ with  $\gamma_{n,m}(X_{i,t}) = 0$. 
Let $X= \bigcup_{i=1}^\infty  \bigcup_{j=\lfloor 1/s\rfloor +1}^\infty X_{i,s-1/j}$ so that $\gamma_{n,m}(X) = 0$. 

Let $F\subset E$ be $\mathcal H^s$-measurable with $\cH ^s(F)>0$; we may assume that $F$ is compact. From \eqref{sum}  $\mu_i (F)>0$ for some $i$. If $V  \in G_{n,m}\setminus X$, then $V  \in G_{n,m}\setminus X_{i,s-1/j}$ for all $j>1/s$, so the $(s-1/j)$-energy 
$I_{(s-1/j)}(\proj_V(\mu_i|_F))< \infty$. Since $\proj_V(\mu_i|_F)$ is supported by $\proj_V F$ it follows from the energy criterion for Hausdorff dimension, see \cite[Theorem 6.9]{Fal3} or \cite[Theorem 8.9]{Mat1}, that
$\dh \proj_V F\geq s-1/j$ for all $j$, so $\dh \proj_V F\geq s$, as required.
\medskip

$({\rm ii})$ This is very similar to $({\rm i})$. When $s>m$,  we again decompose $\cH^s|_{E}$ as in \eqref{sum} with each $\mu_i$ having finite $t$-energy for some $m<t<s$. In this case, following the Fourier transform approach of Kaufman \cite{Kau}, see also \cite[Section 6.3]{Fal3} or \cite[Sections 4.1 and 5.3]{Mat2} or the method of Theorem 9.7 in \cite{Mat1} without the Fourier transform, \eqref{enineq} is replaced by
$$\int_{V \in G_{n,m}}\int_V\big( f_{ \proj_V\mu_i}(x)\big)^2 d{\cal L}^m (x) d \gamma_{n,m}(V) \leq c'_{n,m,t} \int\int\frac{d \mu_{i}(x)d\mu_i(y)}{|x-y|^t}  < \infty,$$
where   $f_{ \proj_V\mu_i}$ is the density of the projected measure $\proj_V\mu_i$, which is absolutely continuous with respect to 
${\cal L}^m$ for almost all $V$. This absolute continuity means precisely that $\mu(F)>0$ implies ${\cal L}^m(\proj_V(F))>0$. 
Letting $X =  \bigcup_{i=1}^\infty X_{i}$, where 
$X_{i} = \{V : \proj_V\mu_i \mbox{  is not absolutely continous} \}$, the conclusion follows in the same way as in $({\rm i})$.

\end{proof}

Falconer and O'Neil in \cite{FO} and Peres and Schlag in \cite{PS} proved independently that if $s >2m$ and $E \subset \mathbb{R}^n$ is an $\mathcal H^s$-measurable set with $0<\cH^s(E) < \infty$, then the interior of $\mbox{\rm proj}_V E$ is non-empty for almost all $V$. The strong version of this is false, at least when $n-m<s$ but quite likely in all cases. To see this, suppose that $s>2m$ and let $E \subset \mathbb{R}^{n}$ be an $\mathcal H^s$-measurable set with $0<\cH^s(E) < \infty$. For a given $V\in G_{n,m}$, let $U$ be the union of a countable dense set of $(n-m)$-dimensional planes orthogonal to $V$. Then for $F=E\setminus U$, $\cH^s(F) > 0$ but $\mbox{\rm proj}_V F$ has empty interior.

The following theorem sharpens Theorem \ref{marthm} (except when $s=m$) by bounding the dimension of the exceptional set of projections:

\begin{theo}\label{marthm1}
Let $1\leq m \leq n-1$ and let $E \subset \mathbb{R}^n$ be an $\mathcal H^s$-measurable set with $0<\cH^s(E) < \infty$, where $s=\dh E>0$. 

$({\rm i})$ If $s\leq m$, there exists a set $X\subset G_{n,m}$ with $\dh X\leq m(n-m-1)+s$ such that for all $V \in G_{n,m}\setminus X$ and all $\mathcal H^s$-measurable sets $F\subset E$ with $\cH ^s(F)>0$,
$\dh \mbox{\rm proj}_V F = s$.\smallskip

$({\rm ii})$  If $s >m$, there exists a set $X\subset G_{n,m}$ with $\dh X\leq m(n-m)+m-s$ such that for all $V \in G_{n,m}\setminus X$ and all $\mathcal H^s$-measurable sets $F\subset E$ with $\cH ^s(F)>0$,
${\mathcal L}^m (\mbox{\rm proj}_V F) > 0$.
\end{theo}

The weaker form of this, when $F=E$, was proved by \cite{Kau} and \cite{KM} in case $({\rm i})$ and by \cite{Fal1} in case $({\rm ii})$, and also presented in \cite[Section 5.3]{Mat2}. These proofs show that if $\mu$ is a measure with finite $t$-energy, then for all $V\in G_{n,m}\setminus X_t$ with $\dh X_t\leq m(n-m-1)+t$, if $t\leq m$ then  $ \proj_V\mu$ has  finite $t$-energy, and for all $V\in G_{n,m}\setminus X_t$ with $\dh X_t\leq m(n-m)+m-t$, if $t>m$ then $\proj_V\mu$ is absolutely continuous. Thus the same argument we used for Theorem \ref{marthm} applies, or it is a corollary of Theorem \ref{PS} below.

One can also consider more general mappings. With similar reasoning we can extend the results for the generalized projections of Peres and Schlag, see \cite{PS}, or \cite[Chapter 18]{Mat2}.  Thus Theorem \ref{PS} will include Theorems \ref{marthm} and \ref{marthm1}, but we preferred to give a separate proof for Theorem \ref{marthm}.

Let
$(\Omega,d)$ be a compact metric space and  
$Q \subset \R^N$  an open connected set. Suppose that the mappings
\begin{displaymath} \Pi_{\lambda} \colon \Omega \to \R^m, \quad \lambda \in Q, \end{displaymath}
are such that the mapping $\lambda \mapsto \Pi_{\lambda}(x)$ is in $C^{\infty}(Q)$ for every fixed 
$x \in \Omega$, and to every compact $K \subset Q$ and every multi-index $\eta=(\eta_1,\dots,\eta_N) \in \N^N, \N=\{0,1,2,\dots\},$ there corresponds a positive constant 
$C_{\eta,K}$ such that
\begin{equation*} |\partial^{\eta}_{\lambda}\Pi_{\lambda}(x)| \leq C_{\eta,K}, \quad \lambda \in K. 
\end{equation*}
Defining
$$\Phi_{\lambda}(x,y) = \frac{\Pi_{\lambda}(x)-\Pi_{\lambda}(y)}{d(x,y)} \mbox{ for }  \lambda \in Q, \ x,y \in \Omega,\  x \neq y,$$
we also assume that the family $\Pi_{\lambda}$, $\lambda \in J$, satisfies \emph{regularity of degree $\beta \geq 0$}:

There exists a positive constant $C_{\beta}$ such that to every multi-index $\eta=(\eta_1,\dots,\eta_N) \in \N^N$ there corresponds a  positive constant $C_{\beta,\eta}$ for which
\begin{equation*} |\Phi_{\lambda}(x,y)| \leq C_{\beta}d(x,y)^{\beta} \quad \Longrightarrow \quad \left| \partial_{\lambda}^{\eta} \Phi_{\lambda}(x,y) \right| \leq C_{\beta,\eta}d(x,y)^{-\beta |\eta|} \end{equation*}
\noindent for $\lambda \in Q$ and $x,y \in \Omega, x\neq y$.

Finally, we assume that $\{\Pi_{\lambda}$, $\lambda \in Q\}$ satisfies \emph{transversality of degree 
$\beta \geq 0$}: 

\begin{equation*} |\Phi_{\lambda}(x,y)| \leq C_{\beta}d(x,y)^{\beta} \quad \Longrightarrow \quad 
\det\big(D_{\lambda} \Phi_{\lambda}(x,y)(D_{\lambda} \Phi_{\lambda}(x,y)^t)\big) \geq C_{\beta}d(x,y)^{2\beta} \end{equation*}
\noindent for $\lambda \in Q$ and $x,y \in \Omega, x\neq y$.

Peres and Schlag  \cite{PS} proved a more general version where only regularity up to some finite degree $L$ is required. Then $L$ appears in the range of the parameters and the analogue of Theorem \ref{PS} follows in the same way. But for simplicity we here restrict to the case $L=\infty$ and Theorem \ref{PS} is the strong version of Theorem 7.3 in \cite{PS}. 

In the case of orthogonal projections  $\mbox{\rm proj}_V, V\in G_{n,m}$, we can take $\beta = 0$ and $N=m(n-m)$, the latter since $G_{n,m}$ is a smooth manifold of dimension $N=m(n-m)$. For various other applications, in particular for Bernoulli convolutions, a positive $\beta$ is needed.

\begin{theo}\label{PS} Under the above assumptions there exists a positive constant $\alpha_0$ depending only on $N$ and $m$ such that the following holds. Let $E \subset \Omega$ be an $\mathcal H^s$-measurable set with $0<\cH^s(E) < \infty$, where $s=\dh E>0$.

$({\rm i})$ If $s\leq m$ and $t \in (0,s-\alpha_0\beta]$, there exists a set $X\subset Q$ with $\dh X\leq N-m+t$ such that for all $\lambda \in Q\setminus X$ and all $\mathcal H^s$-measurable sets $F\subset E$ with $\cH ^s(F)>0$,
$\dh \Pi_{\lambda}(F) \geq t$.\smallskip

$({\rm ii})$ If $s\leq m$ and $t \in (0,s]$, there exists a set $X\subset Q$ with $\dh X\leq N + t - \frac{s}{1+\alpha_0\beta}$ such that for all $\lambda \in Q\setminus X$ and all $\mathcal H^s$-measurable sets $F\subset E$ with $\cH ^s(F)>0$,
$\dh \Pi_{\lambda}(F) \geq t$.\smallskip

$({\rm iii})$ If $s > m$, there exists a set $X\subset Q$ with $\dh X\leq N + m - \frac{s}{1+\alpha_0\beta}$ such that for all $\lambda \in Q\setminus X$ and all $\mathcal H^s$-measurable sets $F\subset E$ with $\cH ^s(F)>0$,
$\mathcal L^m(\Pi_{\lambda}(F))>0$.\smallskip

\end{theo}

\begin{proof}
The basic tool in the proof is Sobolev dimension. For a finite Borel measure $\nu$ on $\R^m$ it is defined by
$$\dim_s\nu=\sup\bigg\{t: \int_{\R^m}|\widehat{\nu}(x)|^2(1+|x|)^{t-m}\,dx<\infty\bigg\},$$
where $\widehat{\nu}$ is the Fourier transform of $\nu$. When $0<t<m$, the integral $\int_{\R^m}|\widehat{\nu}(x)|^2(1+|x|)^{t-m}\,dx$ is comparable to the energy integral $I_t(\nu_{\lambda})$.
When  $t\geq m$, the finiteness of $\int_{\R^m}|\widehat{\nu}(x)|^2(1+|x|)^{t-m}\,dx$ implies that  $\nu$ is absolutely continuous with respect to $\mathcal L^m$. The proofs of these standard facts can be found in \cite{Mat2}.

Let $\mu$ be a finite Borel measure on $\Omega$ with $I_{\alpha}(\mu):=\iint d(x,y)^{-\alpha}\,d\mu x\,d\mu y<\infty$ and let $\nu_{\lambda}$ be the image of $\mu$ under the mapping $\Pi_{\lambda}$. 
We claim that Theorem 7.3 (with $L=\infty$) of \cite{PS} gives a positive constant $\alpha_0$ such that 
\begin{equation}\label{ps}
\dh\{\lambda\in Q: I_t(\nu_{\lambda})=\infty\}\leq N + t - m\ \text{provided}\ \alpha < m\ \text{and}\ t \in (0,\alpha-\alpha_0\beta],
\end{equation}

\begin{equation}\label{ps1}
\dh\{\lambda\in Q: I_t(\nu_{\lambda})=\infty\}\leq N + t - \frac{\alpha}{1+\alpha_0\beta}\ \text{provided}\ \alpha < m\ \text{and}\ t \in (0,\alpha],
\end{equation}

\begin{equation}\label{ps2}
\dh\{\lambda\in Q: \nu_{\lambda}\ \text{is not absolutely continuous}\}\leq N + m - \frac{\alpha}{1+\alpha_0\beta}\ \text{provided}\ \alpha > m.
\end{equation}

Indeed, the inequalities (7.6) and (7.4) in \cite{PS} yield (\ref{ps}) and (\ref{ps1}), as $\{\lambda\in Q: I_t(\nu_{\lambda})=\infty\}\subset 
\{\lambda\in Q: \dim_s(\nu_{\lambda})\leq t\}$, and (7.4) yields (\ref{ps2}), as $\dim_s(\nu_{\lambda})\leq m$ whenever $\nu_{\lambda}$ is not absolutely continuous.

To prove the theorem we write, as in the proof of Theorem \ref{marthm}, $\cH^s|_{E}= \sum_{i=1}^\infty\mu_i$, 
where each $\mu_i$   has finite $\alpha$-energy for $0<\alpha<s$. We apply the above to $\mu = \mu_i$ and let $\nu_{i,\lambda}$ be the image of $\mu_i$ under the mapping $\Pi_{\lambda}$.

For the proof of $({\rm i})$, let $s\leq m$ and $t \in (0,s-\alpha_0\beta)$; the claim for $t=s-\alpha_0\beta$ is easily reduced to this by taking the union of the exceptional sets corresponding to $t=s-\alpha_0\beta-1/j, j=1,2,\dots$. 
Choose $\alpha<s$ with  $t<\alpha-\alpha_0\beta$. Then by (\ref{ps}), 
\begin{equation*}
\dh\{\lambda\in Q: I_t(\nu_{i,\lambda})=\infty\}\leq N + t - m.
\end{equation*}
Now we have that $\nu_{i,\lambda}$   has finite $t$-energy for all $\lambda  \in \Omega\setminus X_{i}$, where  $X_{i}$ is a subset of $\Omega$ with  $\dh X_{i}\leq N + t - m$. 
Let $X= \bigcup_{i=1}^\infty   X_{i}$ so that $\dh X\leq N + t - m$. Then the same argument as in the proof of Theorem \ref{marthm} yields $({\rm i})$.

The proof of $({\rm ii})$ is essentially the same. Let $0<t<\alpha<s$. We now have first by (\ref{ps1}),
\begin{equation*}
\dh\{\lambda\in Q: I_t(\nu_{i,\lambda})=\infty\}\leq N + t - \frac{\alpha}{1+\alpha_0\beta},
\end{equation*} 
and letting $\alpha\to s$,
\begin{equation*}
\dh\{\lambda\in Q: I_t(\nu_{i,\lambda})=\infty\}\leq N + t - \frac{s}{1+\alpha_0\beta},
\end{equation*} 
The rest follows as in (i).

For the proof of $({\rm iii})$ we use the same decomposition $\cH^s|_{E}= \sum_{i=1}^\infty\mu_i$. For each $i$,
applying (\ref{ps2}) to each $\nu_{i,\lambda}$ for every $m<\alpha<s$ and letting $\alpha\to s$,
$$\dh\{\lambda\in Q: \nu_{i,\lambda}\ \text{is not absolutely continuous}\}\leq N + m - \frac{s}{1+\alpha_0\beta}.$$
Thus $\nu_{i,\lambda}$ is absolutely continuous for $\lambda  \in \Omega\setminus X_{i}$, where  $X_{i}$ is a subset of $\Omega$ with  $\dh X_{i}\leq N + m - \frac{s}{1+\alpha_0\beta}$. Letting $X= \bigcup_{i=1}^\infty X_{i}$, we have that  $\dh X\leq N + m - \frac{s}{1+\alpha_0\beta}$ and that $X$ has the desired property as before.
\end{proof}

\section{Sets in hyperplanes}
\setcounter{equation}{0}
\setcounter{theo}{0}

Now we work with hyperplanes in $\mathbb{R}^n$. For $p=(a,b) \in \mathbb{R}^{n-1}\times\mathbb{R}$ let $L(p)=L(a,b)$ denote the hyperplane 
$\{(x,y)\in\mathbb{R}^{n-1}\times\mathbb{R}: y = a\cdot x + b\}$. If $E \subset \mathbb{R}^n$ let $L(E) = \bigcup_{p \in E} L(p)$. 

For $u\in \mathbb{R}^{n-1}$ we write $L_u$ for the vertical line $\{(x,y):x=u\}$ and  
we define 
$$\pi_u: \mathbb{R}^n \to \R,\quad \pi_u(a,b)=a\cdot u + b.$$ 
Then $\pi_u$ is essentially the orthogonal projection $\proj_{l(u)}$ onto the line $l(u)=\{te_u:t\in\R\}$ where $e_u=(1+|u|^2)^{-1/2}(u,1)$. More precisely, $\proj_{l(u)}(p)=(1+|u|^2)^{-1/2}\pi_u(p)e_u$.

Intersections of families of hyperplanes $L(E)$ with vertical lines and projection of sets onto $l(u)$ are related by duality:
\begin{equation}\label{duality}
\mbox{ For } E \subset \mathbb{R}^n,  L_u \cap L(E) = \{u\}\times\pi_u(E).
\end{equation}

We shall use the following simple lemma, see, for example, \cite[Theorem 10.10]{Mat1}:

\begin{lem}\label{lemma} Let $A\subset\R^n$ be a Borel set and let $0<s\leq 1$ be such that $\dh A\cap L_u\geq s$ for $u\in\R^{n-1}$ in a set of positive $\cL^{n-1}$ measure. Then $\dh A\geq s+n-1.$
\end{lem}

Here is our main theorem on the dimension of unions of hyperplanes, generalizing the results of Keleti \cite{Kel}.

\begin{theo}\label{linethm}Let $E\subset\R^n$ be a non-empty Borel set and let $A\subset\R^n$ be a Borel set such that $\cL^{n-1}\big(L(p)\cap A\big)>0$  for all $p \in E$. Then 
$$\dh \big(L(E)\cap A\big)= \dh L(E) = \min\{\dh E + n - 1,n\}.$$
Moreover, if $\dh E>1$, then 
$$\cL^n\big(L(E)\cap A\big)>0.$$
\end{theo}

\begin{proof}

The set $L(E)$ is analytic by \cite[Lemma 2.2(ii)]{Kel}.

By Marstrand's basic projection theorem and  \eqref{duality},  
$$\min \{\dh  E, 1\} =     \dh \pi_u(E) = \dh \big(L_u \cap  L(E)\big)$$
for  $\cL^{n-1}$-almost all $u$. Thus $\dh L(E)\geq \min\{\dh E + n - 1,n\}$ by Lemma \ref{lemma}.

To obtain the opposite inequality, we need to find a basis of coordinates with respect to which we may apply Marstrand's line intersection theorem. We may assume that $\dh  E < 1$. Then $\pi(E)\neq \R^{n-1}$ where $\pi(a,b)=a$ for $(a,b)\in\R^{n-1}$, so there is some $a\not\in\pi(E)$. Rotating the coordinate system, we may then assume that $a=0$ which means that  the normals $e_a\in S^{n-1}$ of all the planes $L(a,b), (a,b)\in E,$ differ from the normal $e_0=(0,\dots,0,1)$ of the coordinate plane $V_0$ of the first $n-1$ coordinates. Hence, writing $E$ as a countable union, we may assume that for some $\delta >0,~ |e_a-e_0|>\delta$ for $(a,b)\in E$. Then if the normal $e_V$ of $V\in G_{n,n-1}$ satisfies $|e_V-e_0|<\delta$ we can write every $L(a,b), (a,b)\in E$, as a graph over $V$:
$$L(a,b)=L_V(c,d):=\{v+\big((c\cdot v)+d\big)e_V: v\in V\},\quad (c,d)\in E_V,$$
where, by simple linear algebra, the new parameter set $E_V\subset V\times \R$ is obtained from $E$ by a smooth transformation, so $\dh E_V= \dh E$. Then we also have $L(E)=L_V(E_V):=\cup_{(c,d)\in E_V}L_V(c,d)$. Given $\epsilon>0$, by Marstrand's line intersection theorem, see \cite[Theorem 10.10]{Mat1}, we can choose such a $V$ so that $\dh L_V(E_V)-n+1-\epsilon\leq\dh (L_{V,u} \cap  L_V(E_V))\leq \dh L(E_V)-n+1$ for $u\in V$ in a set of positive measure; here $L_{V,u}$ is the line $\{u+te_V:t\in\R\}$. It follows that $\dh L(E_V)-n+1=\min \{\dh  E_V, 1\}$. Hence
\begin{equation}\label{eq1}
\dh L(E) = \min\{\dh E + n - 1,n\}.
\end{equation}

\medskip

For the rest, first assume that $E$ is an $s$-set, that is $0<\cH^s(E) < \infty$, where $s=\dh E>0$. Let $A\subset \mathbb{R}^n$ be a Borel set.
For each $u\in \mathbb{R}^{n-1}$, let
\begin{equation}\label{Eu}
E_u = \{p \in E : L_u\cap L(p)\cap A \neq \emptyset\} = \{p \in E : (u,\pi_u(p)) \in A\}.
\end{equation} 
By hypothesis, $\cL^{n-1}\big\{u  \in \mathbb{R}^{n-1}: L_u\cap L(p)\cap A \neq \emptyset\big\}>0$ for all $p \in E$, whence by  Fubini's theorem,
$$(\cL^{n-1} \times \cH^s)\big\{(u,p)  \in \mathbb{R}^{n-1}\times E : L_u\cap L(p)\cap A \neq \emptyset\big\}>0,$$
and so
\begin{equation}\label{eq2}
\cL^{n-1}\big\{u \in  \mathbb{R}^{n-1}: \cH^s(E_u)>0\big\} >0.
\end{equation} 
Applying Theorem \ref{marthm}, it follows that  $\dh \pi_u(E_u)=s\wedge 1$ (where `$\wedge$' denotes `minimum') for almost all  $u$ such that $\cH^s(E_u)>0$, so that
\begin{align*}
0&<\cL^{n-1}\big\{u \in  \mathbb{R}^{n-1}: \dh \pi_u(E_u)=s\wedge 1\big\} \\ 
&= \cL^{n-1}\big\{u \in  \mathbb{R}^{n-1}: \dh \big(L_u\cap L(E_u)\big)=s\wedge1\big\} \\
&= \cL^{n-1}\big\{u \in  \mathbb{R}^{n-1}: \dh \big(L_u\cap L(E)\cap A\big) =s\wedge 1\big\}
\end{align*}
where we have used  duality \eqref{duality} and the definition of $E_u$ \eqref{Eu}.
From Lemma \ref{lemma},  $\dh \big( L(E)\cap A\big) \geq\min\{s+n-1, n\}$.
\medskip

Finally, if $E$ is an arbitrary Borel set, for each $0<s< \dh E$ there exists a compact $E' \subset E$ such that $\dh E' =s$ and $0<\cH^s(E') < \infty$ by a result of Davies \cite{Dav}, see also \cite[Theorem 5.4]{Fal3} or \cite[Theorem 8.19]{Mat1}. Since $\cL^{n-1}(L(p)\cap A))>0$  for all $p \in E'$, we conclude that $\dh \big( L(E)\cap A\big)\geq \dh \big( L(E')\cap A\big) \geq\min\{s+n-1, n\}$ for $s$ arbitrarily close to  $\dh E$. Combining this with \eqref{eq1} completes the proof of the first statement.

Only small changes are needed in the argument to show that $\cL^n\big(L(E)\cap A\big)>0$, if $\dh E>1$. Again it suffices to consider an 
$s$-set $E$ with $s>1$. In this case we need not check separately that $\cL^n\big(L(E)\big)>0$ and we can begin the argument by observing that \eqref{eq2} holds as above. Then by Theorem \ref{marthm},  $\cL^1\big(\pi_u(E_u)\big)>0$  for almost all  $u$ such that $\cH^s(E_u)>0$, which gives
$$\cL^{n-1}\big\{u \in  \mathbb{R}^{n-1}: \cL^1\big(L_u\cap L(E)\cap A\big)>0 \big\}>0,$$
and further, simply by Fubini's theorem, $\cL^n\big(L(E)\cap A\big)>0$.
\end{proof}

We could ask similar questions for $m$-planes in place of hyperplanes, but our method does not work when $1\leq m<n-1$. The case $m=1$ is particularly interesting since it is related to Besicovitch sets. Keleti \cite{Kel} conjectured that in all dimensions any union of line segments has the same Hausdorff dimension as the union of the corresponding lines. This is open when $n\geq 3$. Keleti proved that if true for some $n$ this conjecture would imply that every Besicovitch set in $\R^n$ has Hausdorff dimension at least $n-1$. This would improve the known estimates when $n\geq 5$. Moreover, he showed that if the conjecture is true for all $n$, then every Besicovitch set in $\R^n$ has upper Minkowski dimension $n$, which would be new for all $n\geq 3$.

The problem with $m$-planes, $1\leq m<n-1$, is that as above we are led to families of mappings $\pi_u$, but now the parameter $u$ runs through a space which has smaller dimension than the Grassmannian $G_{n,m}$ and the projection theorem fails. However, one can prove weaker forms 
and apply these to get estimates on $\dh L(E)$ and find conditions which guarantee the positivity of the Lebesgue measure of $L(E)$. Here $L(E)$ is now a union of $m$-planes with a parameter set $E\subset\R^{(m+1)(n-m)}$. Oberlin proved such results in \cite{Ob}. In particular he showed that $\dh E > (m+1)(n-m) - m$ implies $\cL^n(L(E))>0.$

\begin{footnotesize}

{\sc School of Mathematics and Statistics,
University of St Andrews, North Haugh, St Andrews, Fife KY16 9SS, UK,}\\
\emph{E-mail:}
\verb"kjf@st-andrews.ac.uk"

\medskip
{\sc Department of Mathematics and Statistics,
P.O. Box 68,  FI-00014 University of Helsinki, Finland,}\\
\emph{E-mail:} 
\verb"pertti.mattila@helsinki.fi"
\end{footnotesize}
\end{document}